\documentclass[12pt]{amsart}
\usepackage{geometry}                 
\usepackage{amssymb}
\usepackage{amsmath}
\usepackage{amsthm}
\DeclareMathOperator{\mfn}{MULT}
\DeclareMathOperator*{\invlim}{inv\,lim}

\def \ca{cellular automaton}
\def \cas{cellular automata}
\def \odo{odometer}
\def \ZZ{\mathbb Z}
\def \topconj{topologically conjugate}
\def \NN{\mathbb N}
\def \std{space-time diagram}

\begin{document}

\newtheorem*{Th2}{Theorem 2}
\newtheorem*{Th2A}{Theorem 2A}
\newtheorem*{Th2B}{Theorem 2B}
\newtheorem*{theorem}{Theorem}
\newtheorem*{Th1}{Theorem 1}
\newtheorem*{remark}{Remark}
\newtheorem{lemma}{Lemma}

\title {Embedding odometers in  cellular automata}

\author{Ethan M. Coven}
\address{Department of Mathematics, Wesleyan University, Middletown CT 06457-0128}
\email{ecoven@wesleyan.edu}

\author{Reem Yassawi}
\address{Department of Mathematics,
Trent University,
Peterborough ON, Canada K9L 1Z8}
\email{ryassawi@trentu.ca}

\thanks{Some of this work was done while the first author was a guest of Trent University.  He thanks it for its hospitality and support.  The authors thank M.~ Pivato for  introducing them to gliders with reflecting walls.}

\dedicatory{To Micha\l \  Misiurewicz with admiration and affection}



\date{April 15, 2009}

\begin{abstract}
We consider the problem of embedding odometers 
in one-dimensional 
cellular automata. We show that
(1) every odometer can be be embedded in  a  gliders with reflecting walls cellular automaton, which one depending on the odometer, and
(2) an odometer can be embedded in a cellular automaton with local rule  
$x_i \mapsto x_i + x_{i+1}  \mod n$  ($i \in \mathbb Z$), where  $n$  depends on the odometer, if and only if it is  ``finitary.''

\end{abstract}

\maketitle

\section*{Introduction}

An \emph{odometer\/} is the ``$+1$'' map on a countable product of finite cyclic groups.  A (one-dimensional) {\ca\/} $(X,T)$ is a dynamical system defined by a local rule on a closed, $T$-invariant subset of either $A^\NN$ or $A^\ZZ$,
where $A$ is a finite alphabet.
In \cite{CPY} the authors and M.~Pivato
partially solved the ``give me a \ca\ and I will find an \odo\ that can be embedded in it" problem.  In this paper we 
completely
solve the converse problem: ``give me an \odo\ and I will find a \ca\ that it can be embedded in.''

\begin{Th1}
Every \odo\ can be embedded in a  gliders with reflecting  walls \ca.
\end{Th1}

Although finitary odometers (defined in Theorem~2 below) can be embedded in a number of cellular automata \cite{PY}, Theorem~1 identifies a (relatively small) class of cellular automata such that \emph{every\/} odometer can be embedded one of them.

\begin{Th2}
Every finitary \odo\ $(\ZZ(S),+1)$, i.e. one  such that the set of prime divisors of the members of ~$S$ is finite, can be embedded in the one-dimensional, two-sided \ca\  with local rule $x_i \mapsto x_i + x_{i+1}  \mod n$  ($i \in \mathbb Z$), defined on the space of all doubly infinite sequences with entries from~$\ZZ_n$, the ring of integers modulo~$n$,  where $n$ is the product of the primes that divide  infinitely members of~$S$.

Conversely, only finitary  \odo s can be embedded in such  \cas. 
\end{Th2}

\section*{Definitions and Background}

Let $S = (s_1,s_2,\dots)$ be a sequence of integers greater than $1$.
Define
$$\ZZ(S) := \prod_{k\ge1}
\ZZ/s_k\ZZ
\text{\hbox{ }\hbox{ } and \hbox{ }\hbox{ }}
\widetilde\ZZ(S) := \invlim_{k \to \infty}
(\ZZ/s_1s_2\cdots s_k\ZZ, \beta_k),$$
where the binding maps 
$\beta_k : s_1s_2\cdots s_{k+1}\ZZ\to
s_1s_2\cdots s_k\ZZ$
are defined by 
$$z \mapsto z \mod s_1s_2\cdots s_k.$$

\noindent Addition in $\ZZ(S)$ is ``with carrying,'' addition in $\widetilde\ZZ(S)$ is coordinatewise, i.e. without carrying. 
$\ZZ(S)$ and $\widetilde\ZZ(S)$ are
isomorphic, compact, abelian, topological groups~\cite{D}. 

The $+1$ map on $\ZZ(S)$ is defined by
$$z \mapsto z + (1,0,0,\dots)$$ 
and the $+\tilde1$ map on $\widetilde\ZZ(S)$ is defined by
$$z \mapsto z + (1,1,\dots).$$

\noindent$(\ZZ(S),+1)$ and 
$(\widetilde\ZZ(S),+\tilde1)$ are topologically conjugate 
(any topological group isomorphism of $\ZZ(S)$ onto
$\widetilde\ZZ(S)$ that takes~$1$
to~$\tilde1$ is a topological conjugacy)
 and are called the \emph {$S$-adic odometer}.
When $S = (n,n,\dots)$, 
$(\ZZ(S),+1)$ is the well-known
\emph{$n$-adic odometer},
 denoted
$(\ZZ(n), +1)$.

By Theorem ~7.6 of \cite{BS},  a complete topological conjugacy invariant of  
$(\ZZ(S),+1)$ is the \emph{multiplicity function\/} 
$\mfn_S : \{\text{primes}\} \to 
\{0,1,\dots,\infty\}$, defined by
$$\mfn_S(p) := \sum_i \{\max j : p^j \text{ divides } s_i\}.$$
Thus $\mfn_S(p)$ is the total number of times that~$p$ divides members of~$S$.

Throughout this paper a two-sided  \emph{cellular automaton} $(X,T)$ will be  
a dynamical system defined on a closed, $T$-invariant subset of $A^\ZZ$,
where $A$ is a finite alphabet and
$T$ is given by a \emph{local rule}
$\tau : A^{2m+1} \to A$
for some $m \ge 0$ as follows: 
$[T(x)]_i = \tau(x_{i-m},\dots,x_{i+m})$  ($i \in \ZZ$).  $T$~is continuous and commutes with the 
\emph{shift}  
$\sigma: A^{\ZZ} \to A^{\ZZ}$,
defined by  $[\sigma(x)]_i = x_{i+1}$ ($i \in \ZZ$).
When appropriate, we will write
$x \in A^\ZZ$ as $x_L.x_R$,
where the dot separates the negative indices from the non-negative ones.
One-sided \cas\/ are similarly defined.  

When  $A$  has $n$ elements, we may sometimes assume that
$A = \ZZ_n$,  the ring of integers modulo~$n$.  The \ca\ defined on all doubly infinite sequences with entries from~$\ZZ_n$ and local rule
$x_i \mapsto x_i + x_{i+1} \mod n$
$(i \in \ZZ)$
will be denoted~$(\ZZ_n^\ZZ,T_n)$. 
The maps $T_n$ have no memory and so we define one-sided 
cellular automata
$(T_n)_R : \ZZ_n^{\NN_0} \to \ZZ_n^{\NN_0}$
by the same local rule.  Here
$\NN := \{1,2,\dots\}$ is the natural numbers and  
$\NN_0 := \{0,1,2,\dots\}$.

A more geometric class of cellular automata is the class of 
\emph{gliders with reflecting walls\/} cellular automata \cite{M}, Example 6.5.

The alphabet for all these one-sided \cas\ is
$$ 
\{W,L,R,\varnothing\},
$$
where $W$ is a stationary wall, $L$ is a left-moving particle, $R$ is a right-moving particle, and $\varnothing$ is an empty space.

The spaces $X \subseteq A^\NN$ satisfy:
for every $x \in X$,
$x_1 = W$,
$x_i = W$ for infinitely many~$i$,
and between any two consecutive ~$W$ there is exactly one particle.

The local rule for these automata is
\newline\noindent $\bullet$  Walls do not move.
\newline\noindent $\bullet$  If the space immediately to the left of~$L$ is empty,~$L$ and ~$\varnothing$ change places.  If the space immediately to the left of~$L$ is ~$W$, 
$L$ becomes~$R$ but does not move.
\newline\noindent $\bullet$  If the space immediately to the right of~$R$ is empty,~$R$ and ~$\varnothing$ change places.  If the space immediately to the right of~$R$ is ~$W$, $R$ becomes~$L$ but does not move.

\smallskip

For a dynamical system $(X,f)$, where $X$ is a subset of some ~$A^\ZZ$ or ~$A^\NN$,
the \emph{\std\/} of~$(X,f)$ with \emph{seed\/}~$x$
is the array whose $(i,j)^\text{th}$ entry is $[f^j(x)]_i$.
It is a convenient way of visualizing the forward $f$-orbit of~$x$, 
$\{f^j(x):j\ge0\}$.
Here we think of ``increasing time'' as going down.
Space-time diagrams for 
systems on one-sided sequences
are similarly defined, and are convenient ways of visualizing \odo s.

For dynamical systems $(X,f)$ and
~$(\widehat X ,\widehat f)$, we say that 
$(X,f)$ can be \emph{embedded\/} in
~$(\widehat X,\widehat f)$ iff  there is a closed, $\widehat f$-invariant subset
$\widehat X'$ of~$\widehat X$ such that $(X,f)$ is topologically conjugate to
~$(\widehat X',{\widehat f}|_{\widehat X'})$.

\section*{Every  odometer can be embedded in a  gliders with reflecting walls \ca}

Gliders with reflecting walls \cas\    $(X,T)$ are  defined on one-sided infinite sequences with entries from
$\{W,L,R,\varnothing\}$,
with local rules defined in the Definitions and Background section.

\begin{Th1}
Every \odo\ can be embedded in a  gliders with reflecting walls \ca.
\end{Th1}

\begin{proof}
Let $S = (s_1,s_2,\dots)$.

First assume that at least one~$s_i$ is even.  
Since the multiplicity function is a complete topological conjugacy invariant of~$(\ZZ(S),+1)$,  
the order of the ~$s_i$ is irrelevant, so we may assume that~$s_1$ is even.

Consider the set ~$X$ of all  points in
~$\{W,L,R,\varnothing\}^\NN$  of the form
$$W\leftarrow \tfrac{1}{2}s_1 \rightarrow 
W\leftarrow \tfrac{1}{2}s_1s_2 \rightarrow 
W\leftarrow \cdots,
$$
where the gaps 
contain exactly one particle.
     The columns
     of gaps 
in the \std\ 
of a gliders with reflecting walls \ca\ with  any such point as seed     
     are periodic 
 with least periods 
 $s_1,s_1s_2,\dots$.
 
 We show that 
this one-sided  \ca\ is \topconj\ to
$(\widetilde\ZZ(s_1,s_2,s_3,\dots),+\tilde1))$.
Let $\widetilde T$ be the gliders with reflecting walls \ca\ map and
label the gaps, left-to-right, 
$G_1,G_2,\dots$.  Consider the \std\ of
$(X,\widetilde T)$ with seed $\bar x$, defined by
$R$ appears   at the extreme left of each gap.  For $x$ in the forward
$\widetilde  T$-orbit-closure of
~$\bar x$, define
$$
x \mapsto z = (z_1,z_2,\dots)
\in \prod_{k\ge1}
\ZZ/s_1s_2\cdots s_k\ZZ
$$
as follows.  For $i \ge 1$, let $z_i$, $0\le z_i \le s_1s_2\cdots s_i -1$ satisfy
$$x|_{G_i} = \widetilde T^{z_i}(\overset{\longleftarrow\frac{1}{2}s_1s_2\cdots s_i\longrightarrow}{R,\varnothing,\varnothing,\dots,\varnothing}),
$$ 
i.e. $(\overset{\longleftarrow\frac{1}{2}s_1s_2\cdots s_i\longrightarrow}{R,\varnothing,\varnothing,\dots,\varnothing})$ appears in row~$z_i$ in this \std. 
This map is continuous, one-to-one, and commutes with the appropriate actions, and so is a topological conjugacy.

Now assume that all the ~$s_i$ are odd.
In this case the   \ca\ $(X,\widetilde T^2)$, defined on all points of
~$\{W,L,R,\varnothing\}^\NN$  
of the form
$$ W\leftarrow s_1 \rightarrow 
W\leftarrow s_1s_2 \rightarrow 
W\leftarrow \cdots,
$$
in the forward
$\widetilde  T^2$-orbit-closure of
~$\bar x$ (defined as above, $R$ at the extreme left of each gap), 
is topologically conjugate to
$(\widetilde\ZZ(S),+\tilde1))$.
\end{proof}

\section*{An \odo\ can be embedded in a \ca\ with local rule $x_0+x_1$ if and only if it is ``finitary''}

The word ``finitary'' in the title of this section refers to \odo s 
$(\ZZ(S),+1)$ such that the set of prime divisors of the members of~$S$ is finite.

Throughout this section, 
$(\ZZ_n^{\ZZ},T_n)$ will denote the two-sided \ca\ with local rule
$x_i \mapsto x_i + x_{i+1} \mod n$
$(i \in \ZZ)$.
To avoid notational clutter, we may write~$T$ 
rather than~$T_n$ when $n$ is clear.

\begin{lemma}
Let   $\bar x = \dots 000.100 \dots.$
and let $X$ be the forward $T_n$-orbit-closure of~$\bar x$.

\noindent (1) For any $n \ge 2$, $\bar x_R:= 100\dots$ is $(T_n)_R$-fixed. 

\noindent(2) For any $x \in X$,  if some column
$[T_n^j(x)]_i$ $(j\ge0)$ in the 
\std\ of~$(X,T_n)$ with seed~$x$ is periodic with least period~$m$, then the column
immediately to the left,
$[T_n^j(x)]_{i-1}$ $(j\ge0$),  is periodic with least period~$mn'$ for some factor $n'$ of~$n$ ($n' = 1$ or~$n$ is possible).

\noindent(3) For any $n \ge 2$,  
$\bar x$ has an infinite forward $T_n$-orbit.

\noindent(4) For $n = p$ prime, there exist
$1 = k_1 < k_2 < \dots$  such that for every $i \ge 1$,
the columns
$[T_p^j(\bar x)]_i$ $(j\ge0)$, $i =-k_{i+1}+1,\dots,-k_i$  are periodic with least period~$p^i$. 
\end{lemma}

\begin{proof}
Write $T$ in place of ~$T_n$.
\noindent (1) is clear.

\noindent (2)
We prove (2) for $n = p$ prime,
leaving it to the reader to supply the details for the general case.  Suppose that  column~$i$ in the space-time diagram of~$(X,T)$ with seed~$x$ is periodic with least period ~$m$:
$[T^j(x)]_i = 
[T^{j+m}(x)]_i \quad (j \ge 0)$.   
If
$$\sum_{j=0}^{m-1} [T^j(x)]_i \equiv 0 \mod p,$$
then column~$i-1$ is periodic with least period~$m$.
If 
$$\sum_{j=0}^{m-1} [T^j(x)]_i \not\equiv 0 \mod p,$$
then column~$i-1$ is periodic with least period~$pm$.

\smallskip
\noindent (3) If $\bar x$ has a finite forward $T$-orbit, then there exists 
$K\ge0$ such that $x_{-k}=0$ 
for every point~$x$ in this orbit
and for every $k \ge K$.
This contradicts ~(2).  

\smallskip
\noindent (4) follows from (1), (2), and (3). 
\end{proof}

We divide the if and only if statement of Theorem 2 into two separate theorems.

\begin{Th2A} Every finitary \odo\ $(\ZZ(S),+1),$ i.e. one such that the set of prime divisors of the members of ~$S$ is finite, can be embedded in the one-dimensional, two-sided \ca\ $(\ZZ_n^\ZZ,T_n)$ 
 with local rule
$$x_i \mapsto  x_i + x_{i+1} \mod n \quad (i \in \ZZ),$$
where $n$ is the product of the primes that divide infinitely many members of~$S$.
\end{Th2A}

Since the multiplicity function is a  complete topological conjugacy invariant,  every finitary odometer  
is \topconj\ to one of the following two canonical forms:

\noindent (1) the $n$-adic odometer, 
$(\ZZ(n),+1) := (\ZZ(n,n,\dots),+1)$,  where $n$ is the product of distinct primes.

\noindent (2) $(\ZZ(m,n,n,\dots),+1)$, where $m$ and ~$n$ are relatively prime and $n$ is the product of distinct primes.

Theorem~2A follows from  Lemmas~2--7 below.

\begin{lemma}
For $p$ prime and $m\ge2$ such that ~ $p$ is not a factor of~$m$, both 
$(\ZZ(p),+1)$ and
$(\ZZ(m,p,p,\dots),+1)$
 can be embedded in
$(\ZZ_p^\ZZ,T_p).$ 
\end{lemma}

\begin{proof}
Throughout this proof, we write~$T$ in place of~$T_p$.

First we prove the lemma for 
$(\ZZ(p),+1)$.
Consider the space-time diagram of
~$(\ZZ_p^\ZZ,T)$ with seed
$$\bar x = \dots 000.1000 \dots.$$
We show that $T$ restricted to the forward $T$-orbit-closure of~$\bar x$ is \topconj\ to~$(\ZZ(p),+1)$. Recall from Lemma~1, (3) and~(1), that $\bar x$ has an infinite forward $T$-orbit and that 
$\bar x_R := 100\dots$ is $T_R$-fixed.

Define a mapping $x \mapsto z$  from the forward $T$-orbit-closure of~$\bar x$ to~$\ZZ(p)$ as follows.
For $x$ in the forward $T$-orbit-closure of
~$\bar x$,  let
$z = (z_1,z_2,\dots) \in \ZZ(p)$
be such that
$$ 
T^{\sum_{i=1}^k z_i p^{i-1}} (\bar x)
\rightarrow x \quad \text{as } 
k \rightarrow \infty.
$$
\noindent That such a sequence exists follows from Lemma~1. 
(The partial sums of
$\sum_{i=1}^\infty z_i p^{i-1}$
are the rows in the space-time diagram of $(\ZZ_p^\ZZ,T)$
with seed~$\bar x$ at which the
appropriate ``right tails'' of~$x$
appear, so $z$ is well-defined.)
This mapping is continuous, one-to-one, and commutes with the appropriate actions.  Therefore it is a topological conjugacy.

The proof of the lemma for 
$(\ZZ(m,p,p,\dots),+1)$ 
follows the proof for $(\ZZ(p),+1)$
provided we can find a seed
$\bar y = \bar y_L.\bar y_R$ such that
$\bar y$ has an infinite forward $T$-orbit and
$\bar y_R$ is $T_R$-periodic with least period~$m$.
That we can do this is Lemma~4 below.
\end{proof}

\begin{lemma}
$(\ZZ_p^{\NN_0},T_R)$
is \topconj\ to the full one-sided shift  
$(\ZZ_p^{\NN_0},\sigma_L)$,
where  $\sigma_L$ is the left-shift defined by
$[\sigma_L(x)]_i := x_{i+1}$ 
($i\ge 0)$.
\end{lemma}

\begin{proof}
The topological conjugacy $x \mapsto y$
is given by
$y_i := [T_R^i(x)]_0$  ($i\ge0$).
For a more general result, see \cite{BM}.
\end{proof}

\begin{lemma}
Let $m \ge 1$.
There is a point
$\bar y = \bar y_L.\bar y_R$
with an infinite forward $T$-orbit
and such that $\bar y_R$ 
is $T_R$-periodic with
least period~$m$.
\end{lemma}

\begin{proof}

By Lemma~3 there is a $T_R$-periodic point
$\bar y_R = \bar y_0 \bar y_1 \dots$
with least period~$m$.
It follows from Lemma~1(2) 
that column~$0$ in the \std\ of~$T$ (n.b. $T$, not~$T_R$) 
with seed any left extension of
~$\bar y_R$ is periodic.  
So it suffices to show that
$\bar y_R$ has a left extension such that the columns in the appropriate
\std\ have arbitrarily large least periods.

By Lemma~3, for every $k \ge 1$ there are $p^k$ $T_R^k$-fixed points.
For~$k=1$, (since $p^2>p^1$) there exist $\bar y_{-1}$ and $\bar y_{-2}$ such that
$\bar y_{-2}\bar y_{-1}\bar y_0 \dots$ is not $T_R$-fixed.
For~$k=2$, there exist 
$\bar y_{-3}$, $\bar y_{-4}$,  and $\bar y_{-5}$ such that
$\bar y_{-5}\bar y_{-4} \dots$ is not $T_R^2$-fixed.
Continue with $k=3,4,\dots$.
\end{proof}

\begin{lemma}
If $(X,f)$ can be embedded in~$(\widehat X,\widehat f)$ and  $(Y,g)$ can be embedded in~$(\widehat Y,\widehat g$), then
$(X,f) \times (Y,g)$ can be embedded in
$(\widehat X \times \widehat Y, \widehat f \times \widehat g)$.
\qed
\end{lemma}

\begin{lemma}
Let $m,n \ge 2$ be relatively prime.
Then $(\ZZ(mn),+1)$ is
topologically conjugate to
$(\ZZ(m) \times \ZZ(n),(+1,+1))$.
If, in addition, $s \ge 2$ is relatively prime to both ~$m$ and ~$n$, then
$(\ZZ(s,mn,mn,\dots),+1)$ is
topologically conjugate to
$(\ZZ(s,m,m,\dots) \times \ZZ(n),(+1,+1))$.
 
\end{lemma}

\begin{proof}
To prove the first statement it suffices to find a topological group isomorphism  of~$\ZZ(mn)$ onto
$\ZZ(m) \times\ZZ(n)$
that takes 
$(1,0,\dots) \in \ZZ(mn)$ to
$((1,0,\dots),(1,0,\dots)) \in 
\ZZ(m) \times \ZZ(n)$.

Map $\ZZ(mn)$ to
$\ZZ(m) \times \ZZ(n)$ by
$$(z_0,z_1,\dots) \mapsto
  ((z'_0,z'_1,\dots),(z''_0,z''_1,\dots)),$$
where for every $k \ge 0$,
$\sum_{i=0}^k z'_i m^i$ is the beginning of the base~$m$ expansion of $\sum_{i=0}^k z_i (mn)^i$;
similarly for ~$z''$.
This map is well-defined, takes
$(1,0,\dots)$ to
\linebreak
$((1,0,\dots),(1,0,\dots))$,
 and satisfies all the conditions of topological group isomorphism, except possibly ontoness.  To see that it maps $\ZZ(mn)$ onto
$\ZZ(m) \times \ZZ(n)$, notice that it maps the set
$$\{k(1,0,\dots) \in \ZZ(mn): k \ge 0\},$$
which is dense in $\ZZ(mn)$, onto the set
$$\{(k(1,0,\dots),k(1,0,\dots)) \in \ZZ(m) \times \ZZ(n): k \ge 0\}.$$
The latter set is dense in
$\ZZ(m) \times \ZZ(n)$
because $m$ and $n$ are relatively prime.

The proof of the second statement is similar.  We omit the details.
\end{proof}

\begin{lemma}
Let $m,n \ge 2$ be relatively prime.
Then $(\ZZ_{mn}^\ZZ,T_{mn})$ is topologically conjugate to
$(\ZZ_m^\ZZ \times \ZZ_n^\ZZ,T_m \times T_n)$.
\end{lemma}

\begin{proof}
Any ring isomorphism of 
$\ZZ_m \times \ZZ_n$ onto
$\ZZ_{mn}$ 
is a topological conjugacy of
$(\ZZ_{m}^\ZZ \times \ZZ_{n}^\ZZ,T_m \times T_n)$ 
onto~$(\ZZ_{mn}^\ZZ,T_{mn})$.
\end{proof}

\begin{Th2B}
If an \odo\ $(\ZZ(S),+1)$   can be embedded in the one-dimensional, two-sided \ca\  
$(\ZZ_n^\ZZ,T_n)$  with local rule
$$
x_i \mapsto  x_i + x_{i+1} \mod n  \quad (i \in \ZZ),
$$
then $(\ZZ(S),+1)$ is finitary,
i.e. the set of prime divisors of the members of ~$S$ is finite.
\end{Th2B}

\begin{proof}
Suppose that $(\ZZ(S),+1)$ is \topconj\ to~$ (X,T_n)$, where
$X$ is a closed, $T_n$-invariant
subset of ~$\ZZ_n^\ZZ$.
Consider the space-time diagram of 
~$(\ZZ(S),+1)$ with seed
$(0,0,\dots)$.  
Every column is periodic and
for $p$~prime,   
$p$~divides the least period of some column if and only if
$p$~divides some $s \in S$.

It follows from the uniform continuity of the topological conjugacy and its inverse that 
every column in any space-time diagram of~$(X,T_n)$ is periodic. 
For $p$~prime, $p$~divides the least period of some column in a space-time diagram of ~$(X,T_n)$ if and only if
 $p$~divides the least period of some column in the space-time diagram of
~$(\ZZ(S),+1)$ with seed~$(0,0,\dots)$.

The proof is completed by applying the following lemma.
\end{proof}

\begin{lemma}
For any $n \ge 2$,
the set of primes ~$p$  such that $p$ divides the least period of some column in a space-time diagram of
~$(X,T_n)$ is finite.
\end{lemma}

\begin{proof}
Since every column in the space-time diagram of 
~$(\ZZ(S),+1)$ with seed~$(0,0,\dots)$
is periodic, it follows from 
Lemma~1  that every column in any space-time diagram of~$(X,T_n)$ is periodic.
Furthermore, if column~$i$ has least period~$m$, then column~$i-1$ has least period~$n'm$, where $n'$ is a factor of~$n$.

So if a column has least period~$m$, then
any prime that divides the least period of some column to its left also divides~$mn$. Hence the set of all primes that divide the least period of any column is finite.
\end{proof}

\end{document}